\documentclass[12pt,a4paper]{amsart} 
\usepackage{graphicx,amsmath,bm}
\usepackage{amssymb,amscd}
\usepackage[unicode]{hyperref}

\makeatletter

\@addtoreset{equation}{section} \makeatother

\newtheorem{lemma}{Lemma}[section]
\newtheorem{remark}{Remark}[section]
\newtheorem{definition}{Definition}[section]

\newtheorem{theorem}{Theorem}[section]

\newtheorem{corollary}{Corollary}[section]

\newtheorem{conjecture}{Conjecture}[section]
\newtheorem{proposition}{Proposition}[section]
\newtheorem{assumption}{Assumption}[section]
\newtheorem{example}{Example}[section]
\DeclareMathAlphabet{\mathsfsl}{OT1}{cmss}{m}{sl}

\begin{document}

\title [Branching diffusion with interactions]{Branching diffusions with particle interactions}
%%% Resubmission to EJP

\author{J\'anos Engl\"ander and Liang Zhang}

\address{J.E.: Department of Mathematics, University of Colorado\\
 Boulder, CO-80309-0395.}
 
 \address{L.Z.: Oracle Corporation\\
 10075 Westmoor Dr.,  Nr. 200, Westminster, CO-80021.}

\email{Liang.Zhang-2@Colorado.EDU\\ Janos.Englander@Colorado.edu}



\date{\today}
\begin{abstract}
A $d$-dimensional  branching diffusion, $Z$, is investigated, where the linear attraction or repulsion between particles is competing with an Ornstein-Uhlenbeck drift, with parameter $b$ (we take $b>0$ for inward O-U and $b<0$ for outward O-U). This work has been motivated by  \cite{Englander},  where a similar model was studied, but without the drift component. 

We show that the large time behavior of the system depends on the  interaction and the drift in a nontrivial way. Our method provides, \emph{inter alia},   the SLLN for the non-interactive branching (inward) O-U process.

First, regardless of attraction ($\gamma >0$) or repulsion ($\gamma <0$), a.s., as time tends to infinity, the center of mass of $Z$
\begin{itemize}
\item converges to the origin, when $b>0$;
\item escapes to infinity  exponentially fast (rate $|b|$), when $b<0$.
\end{itemize}
We then analyze $Z$ {\it as viewed from the center of mass}, and
finally, for the system as a whole, we show a number of results/conjectures regarding the long term behavior of the system; some of these are scaling limits, while some others concern local extinction.

\end{abstract}

\maketitle
%\tableofcontents

\section{Introduction: Branching motion with drift and self-interaction}
\subsection{Model}
We consider a branching diffusion in $\mathbb{R}^d$, where the motion component is an Ornstein-Uhlenbeck (O-U) process, and dyadic branching occurs in each time unit. (Dyadic branching means precisely two offspring.) In addition, we introduce interaction  between
particles, namely either {\it attraction} or {\it repulsion}.

Let $Z$ denote the process and $Z_t^i$  the $i^{th}$ particle\footnote{We can use an arbitrary labeling, as long as it is independent of the spatial motion.} in time $[m,m+1)$, where $m=0,1,2,...$. As  branching  is unit time, in the time interval $[m,m+1)$
there are $2^m$ particles in total. Without interaction from other particles, $Z^i_t$ is an Ornstein-Uhlenbeck process with drift parameter $b\in \mathbb R$, corresponding to the operator 
$$\frac{1}{2}\Delta -bx\cdot\nabla$$ on $\mathbb R^d$. (Note the negative sign of the drift. It is somewhat unusual, but it fits  our setup better, because of the sign of the interaction parameter $\gamma$, introduced below.) If $b>0$, then we have an `inward'
O-U process; if $b<0$, then we have an `outward' O-U process. If $b=0$, then it is a Brownian motion.

As far as the aforementioned interaction is concerned, let us fix the {\it interaction parameter} $\gamma\ne 0$.
 We assume that the $i^{th}$ particle $Z_t^i$, on the time interval  $[m,m+1)$, `feels' a drift caused by attraction/repulsion of all other particles as
$$\frac{1}{2^m}\sum\limits_{j=1}^{2^m}\gamma\cdot(Z_t^j-\cdot)\, \mathrm{d}t,$$
and so $Z_t^i$ satisfies the following stochastic differential equation:
$$\mathrm{d}Z_t^i = \mathrm{d}W_t^{i,m} - bZ_t^i\, \mathrm{d}t + \frac{1}{2^m}\sum\limits_{j=1}^{2^m}\gamma\cdot(Z_t^j-Z_t^i)\, \mathrm{d}t.$$
 If $\gamma>0$, then this means that particles attract each other, whereas if $\gamma<0$, then this means that
they repel each other.

In the stochastic differential equation above, the $\{W_t^{i,m}\}_{1\le i\le 2^{m}}$ are independent Brownian motions on $[m,m+1)$. In other words, the infinitesimal generator of the $i$th particle in the time interval is 
$$\frac{1}{2}\Delta + \left(\frac{1}{2^m}\sum\limits_{j=1}^{2^m}\gamma\cdot(Z_t^j-x)-bx\right)\cdot\nabla.$$
{\bf Notation.} Throughout the paper, the symbol $\overset{w}\Rightarrow$ (or just  $\Rightarrow$) will denote {\it weak} convergence of finite measures; the symbol $\overset{v}\Rightarrow$ will denote {\it vague} convergence. By a {\it bounded rational rectangle} we will mean a set $B\subset \mathbb R^d$ of the form $B=I_1\times I_2\times\dots \times I_d$, where $I_i$ is a bounded interval with rational endpoints for each $1\le i\le d$. The family of all bounded rational rectangles will be denoted by $\mathcal{R}$. The symbols $X\oplus Y$ will denote the independent sum of the random variables $X$ and $Y$. As usual, $\mathcal{N}(\mu,\sigma^2)$ will denote the normal distribution with mean $\mu$ and variance $\sigma^2$; $\mathtt {Leb}$ will denote $d$-dimensional Lebesgue measure.  Finally, for $z\in\mathbb R$, $\lfloor z \rfloor$ will denote the largest integer which is less than or equal to $z$.

The following criterion will be used in the paper; we omit the standard proof, which follows from the Portmanteau Theorem and the well known condition in Theorem 2.2 in \cite{Billingsley}.
%Recall that by a {\it bounded rational rectangle} we  mean a set $B\subset \mathbb R^d$ of the form $B=I_1\times I_2\times\dots \times I_d$, where $I_i$ is a bounded interval with rational endpoints for each $1\le i\le d$. The family of all bounded rational rectangles is $\mathcal{R}$.

\begin{lemma}\label{PB}
Let $\mu_1, \mu_2, ... $ and $\mu$ be probability measures on $\mathbb{R}^d$ and $\mu<< \mathtt {Leb}$. Then $\mu_n \Rightarrow \mu$ if and only if  $\lim_{n \to \infty}\mu_n(B) = \mu (B)$ for all $B \in \mathcal{R}$.
\end{lemma}

\subsection{Motivation} This  paper has been motivated by \cite{Englander},  where a similar model was studied. There the motion was Brownian motion ($b=0$), and it has been shown  that the center of the system is a Brownian motion, being slowed down such that it tends to a `terminal position' $N$ almost surely, and $N$ is a $d$-dimensional, normally distributed random variable, with mean zero.  If $P^x$ denotes the probability conditioned on $N=x$,  $x\in\mathbb R^d$, then  the following theorem was demonstrated for $\gamma > 0$ (attraction): 
$$ 2^{-n}Z_n(\,\mathrm{d}y) \Rightarrow \left( \frac{\gamma}{\pi}\right)^{d/2} \exp\left(-\gamma |y - x|^2\right)\, \, \mathrm{d}y,\ P^{x}-a.s.,$$ 
as $n \to \infty$ for almost all $x\in\mathbb{R}^{d} $, where $Z(\mathrm{d}y) $ denotes the discrete measure-valued process corresponding to the interacting  branching particle system.  For $\gamma < 0$, a conjecture was stated.

A similar model for superdiffusions has been introduced and studied by Gill recently \cite{G2013} and results analogous to those in \cite{Englander}, were obtained. The toolsets used in the two papers are very different though. Gill's paper utilizes the so-called {\it historical calculus} of E. Perkins.

It should be mentioned that, although our original motivation was to analyze the compound effect of the drift and the interaction, it turns out that  our method yields an elementary proof for the Strong Law  of Large Numbers for the case of a {\it non-interactive} branching (inward) Ornstein-Uhlenbeck process as well. See Example \ref{non.inter.OU}.

Finally, for classical results on  limit theorems for branching particle systems (without interaction), see the fundamental monograph \cite{AH1983}, and the more recent article \cite{EHK2010}.

\subsection{Existence and uniqueness}\label{EandU}
In this section, we  show the existence and uniqueness for this system (process). Actually, it is easy to see\footnote{Otherwise use `concatenating' for the processes.} that we only need to show that, given the initial value in time interval $[m,m+1)$, the system exists and is unique. 

Now,  in the time interval $[m,m+1)$, we can look at the $2^m$ interacting particles (diffusions) as a single $2^md-$dimensional Brownian motion with a drift {\bf d} : $\mathbb{R}^{2^m}\to\mathbb{R}^{2^m}$defined as
$${\bf d}(x_1,x_2,\cdots,x_{d-1},x_{d},\cdots, x_{2^md}): =\gamma(\beta_1,\beta_2,\cdots,\beta_{d-1},\beta_{d},\cdots, \beta_{2^md})^T,$$
where $\beta_k = 2^{-m}\gamma(x_{\overline{k}}+x_{d+\overline{k}}+\cdots + x_{(2^m-1)d+\overline{k}})- (\gamma+b)x_k $. 
Here $\overline{k}- k$ is a multiple of $d$ and $0<\overline{k}\le d$. 
As the drift {\bf d} is Lipschitz, the existence and uniqueness of our system follows from the uniqueness/existence theorem for stochastic differential equations in high dimensions.

\section{The center of mass}
\begin{definition}[Center of mass]\rm
 For $t\in [m, m+1) $, there are $2^m$ particles, denoted by  $\{Z_t^i\}, i = 1,2,\cdots, 2^m$, moving in space. Hence, letting $m:=\lfloor t\rfloor$, we define the center of mass (C.O.M.)  as $$\overline{Z_t}:= \frac{1}{2^m}\sum\limits_{1}^{2^m} Z_t^i.$$
 \end{definition}
 
In this section we are going to show that as $t\to\infty$: 
\begin{itemize}
\item if $b>0$, then the center of mass converges to the origin, no 
matter if attraction or repulsion holds;
\item if $b<0$, then it will tend to infinity with
`speed' $e^{-bt}$.
\end{itemize}
The significance of this result is that the attraction/repulsion for $Z_t^i$ is given by 
$$\frac{1}{2^m}\sum\limits_{j=1}^{2^m}\gamma (Z_t^j-Z_t^i)\, \mathrm{d}t = \gamma(\overline{Z_t} - Z_t^i)\, \mathrm{d}t,$$
where $\overline{Z_t}$ is as above.
Hence, one can replace the interaction between particles by the interaction with the center of mass. Therefore, as a first step, we will study the large time behavior of 
the center of mass $\overline{Z_t}$. 

Before stating our first result, we note that in this section, we will be interested in $a.s.$ and $L^2$ convergence of the center of mass.  Since, it is easy to see that these limits can be verified coordinate-wise, we  assume $d = 1$ for this section. (The reader should keep in mind that the results work for any $d\ge 1$.)

Our main results here will concern the behavior of the center of mass in the attractive/repulsive case.
But we need  some preliminary lemmas first. Below we give two lemmas regarding a general one-dimensional stochastic differential equation 
\begin{equation}\label{general}
\left\{
\begin{array}{ll}
\mathrm{d} {X_t}&= \beta(t) \mathrm{d}{W_t}- b {X_t} \, \mathrm{d}t,\\
X_0&=0,\ a.s.,\\
\end{array}
\right.
\end{equation}
where we assume that  $b>0$ and that $\beta(\cdot)>0$ is locally Lipschitz. Here $\beta(t)$ can be considered a time change of the Brownian part. We assume that $\beta(t)$ converges to $0$ as $t\to\infty$, that is that the Brownian motion is slowing down completely. We then want to determine  the limiting distribution.
\begin{lemma}
\label{Le:var} Let $X$ be the solution of \eqref{general}.

{\sf (a)} If  $\lim_{t\to\infty}\beta(t) = 0 $, then  $\lim_{t\to\infty}{X_t} = 0 \ in \  L^2$.

{\sf (b)} Assume in addition, that  $ \beta(t)$ is decreasing in $t$, and that $$\sum\limits_{m=1}^{\infty}m\beta^2(m)<\infty.$$ Then $\lim_{t\to\infty}{X_t} = 0$  a.s.

\end{lemma}
\begin{proof}
{\sf (a)} Assume, that  $X_t = X(t)$ is of the form  $X(t)=X_1(t)X_2(t)$, with $X_1(0)=0$ and $X_1,X_2$ being of finite variation. Keeping  the product rule for $ \, \mathrm{d} X(t)$ in mind,  set $$\, X_2(t)\mathrm{d} X_1(t) = \beta(t)\mathrm{d}W_t$$ and $$X_1(t)\, \mathrm{d} X_2(t) = -bX(t)\, \mathrm{d}t = - bX_1(t)X_2(t)\, \mathrm{d}t,$$  that is, $\mathrm{d} X_2(t) = -bX_2(t)\, \mathrm{d}t$.

We obtain $X_2(t) = Ce^{-bt},\ C\neq 0$, and thus $\, \mathrm{d} X_1(t) = C^{-1}\beta(t)e^{bt}\mathrm{d}W_t$.
%, that is, $X_1(t) = C^{-1}\int\limits_{0}^{t}\beta(s)e^{bs}\, \mathrm{d} W_s$.
Then $X(t):= e^{-bt}\int\limits_{0}^{t}\beta(s)e^{bs}\, \mathrm{d} W_s$  satisfies \eqref{general}; by uniqueness, it is in fact {\it the} solution to the equation. 

Since $X_t$ is  centered Gaussian, the claim is tantamount to  $\mathtt{Var}(X_t)=e^{-2bt}\int\limits_{0}^{t}\beta^2(s)e^{2bs}\mathrm{d}s\to 0$ (use  It\^o-isometry).
Let $\lim\limits_{t\to\infty}\int\limits_{0}^{t}\beta^2(s)e^{2bs}\mathrm{d}s =\infty$ (otherwise the statement is trivial), and use L'Hospital's rule:
\begin{align*}
\lim\limits_{t\to\infty}\mathtt{Var}(X_t)  &= \lim\limits_{t\to\infty}\frac{\int\limits_{0}^{t}\beta^2(s)e^{2bs}\mathrm{d}s}{e^{2bt}} =\lim\limits_{t\to\infty}\frac{\beta^2(t)}{2b}=0.
\end{align*}

{\sf (b)}  We need to show that for any $\epsilon>0$, we have $$P(\sup\{|X_t|: m\le t<m+1\}>\epsilon,\  \text{i.o.}) = 0.$$ Let $$A_m:=\{\sup\{|X_t|: m\le t<m+1\}>\epsilon\}.$$ Then, by the Borel-Canteli lemma, it is sufficient to show that
 \begin{equation}\label{BC.cond}\sum\limits_{m=1}^{\infty}P(A_m)< \infty.
 \end{equation}
Denote $Y_t :=e^{bt}X_t=\int\limits_{0}^{t}\beta(s)e^{bs}\, \mathrm{d} W_s $, and note that $Y_t$ is a 
martingale as it is an It\^o integral. Thus $|Y_t|$ is a submartingale.

We have $P(A_m)<P(\sup\{|Y_t|: m\le t<m+1\}>\epsilon e^{bm})$.
By Doob's inequality,
$$P(\sup\{|Y_t|: m\le t<m+1\}>\epsilon e^{bm})<\frac{E(|Y_{m+1}|^2)}{\epsilon ^2 e^{2bm}}.$$
As $E(Y_t)= 0$, we have
$$E(|Y_t|^2) = \mathtt {Var}(Y_t) = \int\limits_{0}^{t}\beta^2(s)e^{2bs}ds,$$
and thus $$\frac{E(|Y_{m+1}|^2)}{\epsilon ^2 e^{2bm}} = \frac{\int\limits_{0}^{m+1}\beta^2(s)e^{2bs}ds}{\epsilon ^2 e^{2bm}}.$$
It remains to prove that
$$ \sum\limits_{m=1}^{\infty}e^{-2bm}\int\limits_{0}^{m+1}\beta^2(s)e^{2bs}ds<\infty.$$ 
For $e^{-2bm}\int\limits_{0}^{m+1}\beta^2(s)e^{2bs}ds$, we break up the expression into two parts:
\begin{align*}
e^{-2bm}\int\limits_{0}^{m+1}\beta^2(s)e^{2bs}ds &= e^{-2bm}\int\limits_{0}^{(m+1)/2}\beta^2(s)e^{2bs}ds+e^{-2bm}\int\limits_{(m+1)/2}^{m+1}\beta^2(s)e^{2bs}ds\\
&=:I_1^m+I_2^m.
\end{align*}
We show now that both $ \sum\limits_{m=1}^{\infty} I_1^m $  and $ \sum\limits_{m=1}^{\infty} I_2^m $  are finite. 

\underline{$I_1^m$ summable:} 
$$I_1^m =  e^{-b(m-1)}\int\limits_{0}^{(m+1)/2}\beta^2(s)e^{2bs-bm-b}ds \le  e^{-(m-1)}\int\limits_{0}^{(m+1)/2}\beta^2(s)ds.$$ 

As $\beta$ decreases to $0$, there is  a constant $C$ such that $\beta(s)<C$ for all $s\ge 0$. Then $\int\limits_{0}^{(m+1)/2}\beta^2(s)ds < C^2(m+1)/2 $.
For large $m$,  $m-1>(m+1)/2$, and so
$$I_1^m\le C^2e^{-b(m-1)}(m-1),$$ 
yielding that $ \sum\limits_{m=1}^{\infty} I_1^m\le C^2\sum\limits_{m=0}^{\infty}me^{-bm}=C^2(e^{b}+e^{-b}-2)^{-1} < \infty$.

\underline{$I_2^m$ summable:} 
\begin{align*}
I_2^m& = e^{-2bm}\int\limits_{(m+1)/2}^{m+1}\beta^2(s)e^{2bs}ds \le e^{2b}\int\limits_{(m+1)/2}^{m+1}\beta^2(s)ds\\
&\le e^{2b}(m+1)/2\cdot\beta^2 ((m+1)/2).
\end{align*}
Note that $(m+1)/2\le 2\lfloor(m+1)/2\rfloor$ for  $m\geq 1$, and, since $\beta$ is a decreasing function, one has
\begin{align*}\sum\limits_{m=1}^{\infty} I_2^m&\le \sum\limits_{m=1}^{\infty} e^{2b}(m+1)/2\cdot\beta^2 ((m+1)/2)\\&\le 2 e^{2b}\sum\limits_{m=1}^{\infty}\lfloor(m+1)/2\rfloor\beta^2(\lfloor(m+1)/2\rfloor)\le 4 e^{2b}\sum\limits_{n=1}^{\infty}n\beta^2(n).
\end{align*}
 Since, by assumption, $\sum\limits_{m=1}^{\infty}m\beta^2(m)<\infty$, we have 
$ \sum\limits_{m=1}^{\infty} I_2^m < \infty.$

Since   $I_1^m$ and $I_2^m$ are summable, the summability condition
\eqref{BC.cond} indeed holds.
\end{proof}

Finally, here is a lemma that describes the COM as a process.
\begin{lemma}[SDE for COM]
 On $[0,\infty)$, the process $\overline{Z}$ satisfies the stochastic differential equation $\mathrm{d} \overline{Z_t}= \beta(t) \mathrm{d}{W_t}- b \overline{Z_t} \, \mathrm{d}t$ with $\beta(t) := 2^{-m/2}$, for $t\in [m,m+1)$, where $W$ is a standard Brownian motion. Consequently,
\begin{equation}\label{integral.repr}
Z_t=e^{-bt}\int_0^t \beta (s) e^{bs}\mathrm{d}W_s.
\end{equation}
\end{lemma}
\begin{proof}
\noindent Consider the time interval $[m,m+1)$, and recall  the definition of the center of mass: $\overline{Z_t} = 2^{-m}\sum\limits_{i=1}^{2^m}Z^{i,m}_t$. For each $i$, the particle $ Z^{i,m}_t$'s motion satisfies
the stochastic differential equation $$\mathrm{d} Z^{i,m}_t = \mathrm{d}W^{i,m}_t + \left( \gamma 2^{-m}\sum\limits_{j=1}^{2^m}(Z^{j,m}_t - Z^{i,m}_t)-bZ^{i,m}_t\right)\, \, \mathrm{d}t,$$
where $\gamma$ is the interaction coefficient, and $b$ is the drift part of the Brownian motion. In our case, we consider 
$b > 0$. 

Taking averages on both sides,  the center of mass $\overline{Z_t}$ will thus satisfy
 $$\mathrm{d} \overline{Z_t} = 2^{-m}\sum\limits_{i=1}^{2^m}\mathrm{d}W^{i,m}_t -   b  \overline{Z_t}\, \, \mathrm{d}t.$$
 \noindent As the Brownian components of different particles are independent, Brownian scaling yields that $$2^{-m}\sum\limits_{i=1}^{2^m}\mathrm{d} W^{i,m}_t = 2^{-m/2} \mathrm{d} \widetilde{W_t},$$
where  $\widetilde{W_t}$ is standard Brownian motion in the time interval $[m,m+1)$.
We thus have 
$$\mathrm{d}  \overline{Z_t}= 2^{-m/2} d\widetilde{W_t}- b  \overline{Z_t} \, \mathrm{d}t.$$
\\
Hence, in the time interval $[m,m+1)$, we have an Ornstein-Uhlenbeck process, while on $[0,\infty)$, the process $\overline{Z}$ satisfies the  general stochastic differential equation in the statement.
\end{proof}

After these preparations, we now turn to the attractive case.
\begin{theorem}[COM; Attraction]
\label{Th:center}
\noindent If $b>0$, then $\lim_{t\to\infty}\overline{Z_t}=0 \ a.s.$ \\
\end{theorem}
\begin{proof}
Recall the stochastic differential equation satisfied by COM. In order to prove the theorem, it is sufficient to check that  the conditions of  Lemma \ref{Le:var} are satisfied. 

Clearly, $\beta$ is decreasing and locally Lipschitz, as  $\beta(t) = 2^{-m/2}$  for $t\in [m,m+1)$,
and furthermore,
$$\sum\limits_{m=1}^{\infty}m\beta^2(m) = \sum\limits_{m=1}^{\infty}m2^{-m} <\infty.$$
Thus,  $\lim_{t\to\infty}\overline{Z_t}=0 \ a.s.$
\end{proof}

\medskip
 For the repulsive case ($b<0$), we have the following theorem.
\begin{theorem}[Exponential escape of the COM for repulsion]\label{escape}
For $b<0$, $\lim\limits_{t\to \infty}e^{bt}\overline{Z_t}=\mathcal{N}$  a.s., where $\mathcal{N}$ is a normal variable with mean zero and  $$\mathtt{Var}(\mathcal{N})=\frac{1-e^{2b}}{|b|(2-e^{2b})} \cdot I_d.$$
(Here $I_d$ is the identity matrix.)
\end{theorem} 
\begin{proof}
By the independence of the coordinate processes, it is enough to consider $d=1$.
In order  to show the existence of the limit and to identify it, we are going to utilize the Dambis-Dubins-Schwarz Theorem.
We will use the shorthand $\overline {X_t}:=e^{bt}\overline {Z_t}.$

%%%% Center limit for b<0
%\begin{claim}[Limit Law]
%$\lim\limits_{t\to \infty}\overline {Z_t}/ e^{-bt} = N a.s.$, where $N$ is a normal random variable with mean $\mu = 0$ and $\sigma= \frac{e^{2b}-1}{b(2-e^{2b})}$.
%\end{claim}

More precisely, we are going to show that there exists a Brownian motion $B$  on the same probability space where $Z$ is defined, such that   $\overline{X_t} = B_{s(t)},\ P$-a.s. Here $t\mapsto s(t)$ is a deterministic time-change of $t$, mapping $[0,\infty)$ to a finite interval, satisfying that $\lim_{t\to \infty} s(t) = T,$
where \begin{equation}\label{def.T}
T=T(b):= \frac{1-e^{2b}}{|b|(2-e^{2b})}.
\end{equation}
Consequently, we will have that
 $$ \lim\limits_{t\to \infty}\overline {X_t}=\lim\limits_{t\to \infty} B_{< \overline {X}>_t} = B_{\lim\limits_{t\to \infty}< \overline {X}>_t}= B_T.$$

To achieve all these,  recall first that by \eqref{integral.repr}, $\overline {X_t}=\int\limits_{0}^{t}\beta(s)e^{bs}\, \mathrm{d} W_s$, and thus, it is a continuous martingale. Therefore by the Dambis-Dubins-Schwarz Theorem (see e.g. Theorem V.1.6 in \cite{RY2004}), $\overline {X_t}$ is a time-changed Brownian motion: 
$$ \overline {X_t}= B_{< \overline {X}>_t},\ a.s.$$
where $< \overline {X}>$ denotes the increasing process for $\overline {X}$. Since the increasing process is deterministic in this case, we have that
$$s(t):=< \overline {X}>_t=\mathrm{\mathtt {Var}}( \overline {X}_t) =\int\limits_{0}^{t}\beta^2(s)e^{2bs}\, \mathrm{d} s,$$
where $\beta(s) := 2^{-m/2}$ for $s\in [m, m+1)$. Thus,   $\overline {X_t} = B_{s(t)}$, almost surely, and furthermore,
$$\lim\limits_{t\to \infty}s(t)= \sum\limits_{m=0}^{\infty }\int\limits_{m}^{m+1}2^{-m}e^{2bs}\, \mathrm{d} s
 = \sum\limits_{m=0}^{\infty }2^{-m}\cdot\frac{e^{2b(m+1)}-e^{2bm}}{2b}.$$
 
\noindent  To evaluate the infinite sum, one can  use Abel's (summation  by part) formula, which leads to:
$$
 \lim\limits_{t\to\infty}s(t) = \sum\limits_{m=0}^{\infty }2^{-m}\cdot\frac{e^{2b(m+1)}-e^{2bm}}{2b}
  = \frac{e^{2b}-1}{b(2-e^{2b})}=T,
$$
completing the proof.\end{proof}
We note that without the Dambis-Dubins-Schwarz Theorem, much more elementary, standard arguments  still prove  the existence of the almost sure limit, but only along certain `discrete time skeletons.' 

\begin{remark}[Exponential speed of C.O.M.]
\rm As $\lim\limits_{t\to \infty}e^{bt}\overline {Z_t}$ exists a.s. and $ e^{|b|t} \to \infty $, the point $\overline {Z_t}$ will tend to infinity, almost surely, with  `speed' $ e^{|b|t}$ in the sense that $\overline {Z_t} \thickapprox e^{|b|t}\cdot \mathcal{N}$. Furthermore,  even in higher dimensions, it is clear by symmetry considerations that the angular component of  $\overline {Z_t}$ will be uniformly distributed. 

Finally, it is easy to see that $\lim\limits_{b\to 0} T(b) = 2$, in accordance with the already studied driftless case.$\hfill\diamond$
\end{remark}
\section{The system as viewed from the center (`relative system')}
\noindent Having described the motion of the center of mass $\overline{Z_t}$, in order to study the whole system, we need to investigate the `relative system', that is the system as viewed from  $\overline{Z_t}$.

\begin{definition}[Relative system]{\rm 
Denote $Y_t^i := Z_t^i - \overline{Z_t}$. The particle system $\{Y_t^i\}_{i=1}^{2^{\lfloor t\rfloor}}$ will be called the \emph {relative system}, or \emph {the system, as viewed from the center of mass}. }
\end{definition}
We focus on the behavior of the relative system in this section. We will use the shorthand $\sigma_m^2:=1-2^{-m}$.

First, we want to determine the stochastic differential equation for $Y_t^i$. It can be obtained by direct computation, as follows. 
Fixing the time interval $[m,m+1)$, recall that  for each $i$, the particle $ Z^{i,m}_t$'s motion satisfies
 $$\mathrm{d} Z^{i,m}_t = \mathrm{d}W^{i,m}_t + \left( \gamma 2^{-m}\sum\limits_{j=1}^{2^m}(Z^{j,m}_t - Z^{i,m}_t)-bZ^{i,m}_t\right)\, \, \mathrm{d}t,$$ while  $\overline{Z_t}$ satisfies that
 $$\mathrm{d} \overline{Z_t} = 2^{-m}\sum\limits_{i=1}^{2^m}\mathrm{d}W^{i,m}_t -   b  \overline{Z_t}\, \, \mathrm{d}t.$$
Subtracting the second equation from the first, one has
\begin{align*}
&  \mathrm{d}Y_t^i=\mathrm{d}(Z_t^i - \overline{Z_t})= \\
&\ \ \  \sigma_m^2 \mathrm{d}W_t^{i,m}+ \sum_{j\ne i} - 2^{-m} \mathrm{d}W_t^{j,m} +\left(\gamma \overline{Z_t}-\gamma Z_t^{i,m}-bZ_t^{i,m}+b \overline{Z_t}\right)\, \mathrm{d}t=\\
&\ \ \ \sigma_m^2 \mathrm{d}W_t^{i,m}+ \sum_{j\ne i} - 2^{-m} \mathrm{d}W_t^{j,m} -(\gamma + b) Y_t^i \, \mathrm{d}t.
\end{align*}
As $\{W_t^{i,m}\}$ are independent standard Brownian motions, a short computation shows  that 
$\sigma_m^2 W_t^{i,m} \bigoplus_{j\ne i} - 2^{-m} W_t^{j,m}$ is a Brownian motion with variance $\sigma_m^2 t$ at time $t>0.$ Hence,
$$\mathrm{d}Y_t^i= \sigma_m \mathrm{d}\widetilde{W}_t^i -(\gamma + b) Y_t^i \, \mathrm{d}t,$$
where $\widetilde{W}_t^i$ is a driving standard Brownian motion for $Y_t^i$, such that
$$\sigma_m \widetilde{W}_t^i =\sigma_m^2 W_t^{i,m} \bigoplus_{j\ne i} - 2^{-m} W_t^{j,m}.$$
When $ t \to \infty$, (i.e., $m\to \infty$), the process $Y^i$ will asymptotically satisfy the  equation 
\begin{equation}\label{asympt.eq}
\mathrm{d}Y_t^i=  \mathrm{d}\widetilde{W}_t^i -(\gamma + b) Y_t^i \, \mathrm{d}t,
\end{equation}
 yielding that, for large times, the motion of $Y_t^i$ is very close to the one governed by \eqref{asympt.eq}, namely, to an
\begin{enumerate}
\item[{\sf (i)}]     inward O-U process, if $\gamma + b > 0$; 
\item[{\sf (ii)}]   outward O-U process, if $\gamma + b <0$;
\item[{\sf (iii)}]    Brownian motion,  if $\gamma + b = 0$.
\end{enumerate}

As a next step, we need to study the correlation between the particles of $\{Y_t^i\}$ for $t>0$. As $\sum_i\widetilde{W}_t^i = 0$, they are obviously not independent. 
 
 First we determine the `degree of freedom' of $\{\widetilde{W}_t^i\}$. 
Similarly to  \cite{Englander}, one can show that the degree of freedom of  $\{\widetilde{W}_t^i\}$
is $2^m-1$. To explain what this means, 
fix $m\ge 1$ and for $t\in[m,m+1)$ let
$Y_t:=(Y_t^1,...,Y_t^{2^{m}})^T$, where $()^T$ denotes transposed. (This is a vector of length $2^m$
where each component itself is a $d$-dimensional vector; one can actually view it as a $2^m\times d$
matrix too.) We then have
$$\mathrm{d} Y_t=\sigma_m \mathrm{d}\widetilde W^{(m)}_t-\gamma Y_t \mathrm{d}t,$$
where $$\widetilde W^{(m)}= \left(\widetilde W^{m,1},...,\widetilde W^{m,2^{m}}\right)^T$$ and
$$\widetilde W_{\tau}^{m,i}=\sigma_m^{-1}\left(W_{\tau}^{m,i}-2^{-m}\bigoplus_{j=1}^{2^m}
W_{\tau}^{m,j}\right),\ i=1,2,...,2^m$$ are mean zero, correlated
Brownian motions.

Just like in subsection \ref{EandU}, here we can also consider $Y$ as a single $2^md$-dimensional
diffusion.  Each of its components is an Ornstein-Uhlenbeck process with asymptotically unit diffusion
coefficient.

By independence, it is enough to consider the $d=1$ case.

Let us first describe the distribution of $\widetilde W^{(m)}_t$ for $t\ge 0$ fixed. Recall that $\{W_s^{m,i},\ s\ge 0;\
i=1,2,...,2^m\}$ are independent Brownian motions. By definition, $\widetilde W^{(m)}_t$ is a
$2^m$-dimensional multivariate normal:
\begin{eqnarray*}\label{matrix} \widetilde W^{(m)}_t=\sigma_m^{-1}\cdot\left( \begin{array}{ccc}
1-2^{-m} & -2^{-m} & ... \  -2^{-m} \\
-2^{-m} & 1-2^{-m} & ... \  -2^{-m} \\
.\\
.\\
.\\
-2^{-m} & -2^{-m} & ...  \  1-2^{-m} \end{array} \right)W_t^{(m)}\\
\hfill=:\sigma_m^{-1}\mathbf{A}^{(m)} W_t^{(m)},
\end{eqnarray*}
where $W_t^{(m)}=(W_t^{m,1},...,W_t^{m,2^{m}})^T$,
yielding
$$\mathrm{d} Y_t= \mathbf{A}^{(m)} \mathrm{d}W_t^{(m)}-\gamma Y_t \mathrm{d}t.$$
Since we are viewing the system from the center of mass, $\widetilde W^{(m)}_t$ is a \emph{singular} multivariate normal
and thus $Y$ is a degenerate diffusion. The `true' dimension of $\widetilde W^{(m)}_t$  is  r$(\mathbf{A}^{(m)})$.
Then the same argument as in \cite{Englander}, yields
that 
rank$(\mathbf{A}^{(m)})=2^m-1,$
and the above comment about the degrees of freedom should be understood in this sense.

Moreover, the driving Brownian motions $\{\widetilde{W}_t^i\}$ will be exactly the same as in \cite{Englander},
and thus they have asymptotically vanishing correlation (see Remark 12 in \cite{Englander}).

The relative system thus coincides with the driftless one in \cite{Englander}, if $\gamma$ is replaced by $\gamma+b$.

%% More about The relative system viewed from the center %%
\section{A useful transformation: the process $Z^{\Delta}$}
We first make an important observation, making the last sentence of the previous section more general:
we notice  that  $\gamma$ and $b$ are `interchangeable' in the following sense.
\begin{lemma}[Interchangeable coefficients]\label{interchange}\rm
Suppose that we have two branching particle systems, and $Y$ and $\mathcal{Y}$ represent the relative systems for them. 
Denote by $b_1, \gamma_1$ and $b_2, \gamma_2$  the corresponding coefficients of   $Y$ and $\mathcal{Y}$. 
Assume that  $b_1+\gamma_1 = b_2+\gamma_2$. Then the laws of $Y$ and $\mathcal{Y}$ are the same. 
\end{lemma}
\begin{proof}
Fix $m\ge 0$ and  $i = 1,2,\cdots,2^m$. Then $Y^i$ and $\mathcal{Y}^i$ satisfy the same stochastic differential equation
$$\mathrm{d}Y_t^i=  (1-2^{-m})^{1/2}\mathrm{d}\widetilde{W}_t^i -\gamma Y_t^i \, \mathrm{d}t $$ in the
interval $t \in [m,m+1)$, where  $\gamma = b_1 + \gamma_1 = b_2+\gamma_2$.  Thus, the fact that  single particles have the same law in the two systems, follows by induction, along with the existence and uniqueness of the solution for stochastic differential equations. (We know that  for $m = 0$ they start with the same initial value.)

The fact that even the {\it joint} distributions of the two particle systems agree, follows the same way as we proved existence and uniqueness for the model in subsection \ref{EandU}, except that now the independent driving Brownian motions must be replaced by 
$\sigma_m^{-1}\mathbf{A}^{(m)} W_t^{(m)}$ in $[m,m+1)$ (recall \eqref{matrix}).
Since the piecewise Lipschitz-ness of the coefficients is preserved, the existence and uniqueness theorem is still in force.
\end{proof}
We now define a transformation which plays a crucial role in this paper.
\begin{definition}[$\Delta$-transformation] \rm Consider $Z$ with $\gamma$ and $b$ given and let $\mathcal{Z}=Z^{\Delta}$ be another system with 
parameters 
\begin{eqnarray*}
\gamma^{\Delta}&:=\gamma-\Delta;\\ b^{\Delta}&:=b+\Delta.
\end{eqnarray*} 
Since $b+\gamma=b^{\Delta}+\gamma^{\Delta}$, we know by Lemma \ref{interchange} that the corresponding relative systems $Y$ and $\mathcal{Y}$ have the same law.

Consider $Z$ and $\mathcal{Z}$ on the same probability space. Then
\begin{equation}
Z_t(B)=Y_t(B-\overline{Z}_t)\overset{d}=\mathcal{Y}_t(B-\overline{Z}_t)=\mathcal{Z}_t(B-(\overline{Z}_t-\overline{\mathcal{Z}}_t)),\label{nice.match}
\end{equation}
for $B\subset \mathbb{R}^d$ Borel and $t\ge 0$. In fact, 
\begin{equation}\label{process.match}
Z\ \text{has the same law as the process}\ t\mapsto  \mathcal{Z}_t(\cdot-D_t),
\end{equation}
where $D_t:=\overline{Z}_t-\overline{\mathcal{Z}}_t$. The behavior of $D_t$ for large times depends on the signs of $b$ and $b+\gamma$. E.g. if they are both positive, then $D_t$ tends to the origin a.s.

In particular, with appropriate transformations we can `knock out' either the interaction or the O-U drift:

{\bf (a) Representation with non-interactive system:} pick $\Delta:=\gamma$. That is, let the non-interactive process $\mathcal{Z}$ correspond to $\gamma^{\Delta}=0$ and $b^{\Delta}:=b+\gamma$. Then \eqref{process.match} gives a remarkable link between the interactive and the non-interactive models. This connection is reminiscent of the one found in \cite{G2013} (see Remark 3.2 there).

{\bf (b) Eliminating the O-U drift:} pick $\Delta:=-b$. That is, the motion component of $Z^{\Delta}$ is just a Brownian motion, similarly to \cite{Englander}. Then \eqref{process.match} gives a  link between our model and the driftless one studied in \cite{Englander}.
\end{definition}

\section{Outline of the strategy of the rest of the proofs}
In light of the previous section, we could choose to base the analysis of the relative system on the corresponding result in \cite{Englander} when $b+\gamma>0$
(by eliminating the drift -- see part (b) in the previous section), or on the results on the global system in Example 11 of \cite{EHK2010} when $b+\gamma<0$ (by eliminating the interaction -- see part (a) in the previous section).

In the second case, we should handle the problem that the setting is different in \cite{EHK2010} in that the branching is not unit time but rather exponential.

On top of that, the method of the proof in both \cite{Englander} and \cite{EHK2010} requires the introduction of two non-trivial auxiliary functions, related to the model.

Instead of choosing one of the paths alluded to above, we decided to give a completely elementary proof in the next section  for the Strong Law for the relative system in our case, when $b+\gamma>0$. It does not use the complicated machinery of \cite{Englander} or \cite{EHK2010}, and it is done in unit time. The proof only uses some calculations involving the most recent common ancestors of particles and some covariance estimates.

In particular, it gives a new, elementary proof for the Strong Law for the global system, for the case of a non-interactive branching (inward) Ornstein-Uhlenbeck process. See Example \ref{non.inter.OU}.

\section{Proof of SLLN for the relative system when $b+\gamma>0$}
\subsection{General comments}
Recall that $\sigma_m^2:=1-2^{-m}$, and that for the relative system $Y$,
$$\mathrm{d}Y_t^i=  \sigma_m \mathrm{d}\widetilde{W}_t^i -(b+\gamma) Y_t^i \, \mathrm{d}t $$ for $i = 1,2,3,\cdots,2^m$ in the
interval  $t \in [m,m+1)$. Recall also that  $\sigma_m d\widetilde{W}_t^i = \sigma_m^2 dW_t^{i,m} \bigoplus_{j\ne i} - 2^{-m} dW_t^{j,m} $.

Assuming $b+\gamma>0$, our goal is to find  $\lim_{m\to\infty}\frac{1}{2^m}\sum\limits_{i=1}^{2^{m}} Y_t^i(B)$ for a generic Borel set $B\subset \mathbb R^d$.  (Here, we consider $Y_t$ as a random measure, and that is why we may write  $Y_t^i(B) = 1_B(Y_t^i), i = 1,2,3,\cdots,2^m$.) For $d=1$, this will be achieved in Theorem \ref{SLLN.thm} and 
subsequently, it will be upgraded to higher dimensions; before these, we will prove several preparatory results.

\begin{remark}\rm
When taking the limit, we will just  consider  integer times. This is  somewhat weaker than considering continuous times, however, since the model is about unit time branching, we did not have sufficient motivation to go into the technical details as to how one upgrades the limit along integer times to a limit along continuous times. (There are existing techniques though, going back to the work of Assmussen and Hering, see \cite{AH1983,EHK2010}.)$\hfill\diamond$
\end{remark}

As mentioned above, for simplicity we will first treat $d = 1$, and then show that the main result we got also works  for high dimensions.

Next, let us sketch our {\it strategy} of the investigation:
\begin{enumerate}
\item Find the correlation between the particle positions $Y_m^i, 1\le i\le 2^m;m\ge 1$.
\item Use the correlations in (1) to control the correlation between $Y_m^i(B)$, where $B\subset \mathbb R^d$. 
\item Establish the Strong Law of Large Numbers for $Y$, that is, find a measurable
function $0\le f$ such that for $B\subset \mathbb R^d$ Borel,
$$\lim\limits_{m\to \infty}\frac{1}{2^m}\sum\limits_{i=1}^{2^m}Y_m^i(B) = \int\limits_{B}f(x) \, \mathrm{d}x, a.s.$$
\end{enumerate}
\begin{assumption}[No drift]\label{nodriftass}\rm 
In this section, when  studying the relative system, we will assume that $b=0$ and $\gamma>0$.
We can do this without the loss of generality, since given $\gamma$ and $b$, one may apply a $\Delta$-transformation with $\Delta:=-b$, that is, eliminate the O-U drift. Then $b^{\Delta}=0$ and $\gamma^{\Delta}=b+\gamma$, and $Y$ and $\mathcal{Y}$ have the same law by Lemma \ref{interchange}. 
\end{assumption}

\subsection{Crucial estimates}
\noindent Now let us focus on  the distribution of the random variables $Y_t^i, 1\le i\le 2^{\lfloor t\rfloor},t\ge 0$. Since
$dY_t^i=\sigma_m d\widetilde{W}_t^i -\gamma Y_t^i \, \mathrm{d}t$, and since  for the general differential equation 
$$ dY_t = a d W_t - r Y_t\, \, \mathrm{d}t,$$ ($a,r>0$) with initial value $Y_0$, the solution is $Y _t = e^{-r t}\left[ \int\limits_{0}^{t}e^{r s}d(aW_s)+Y_0\right],$
it follows, by conditioning on $Y_m^i$, that 
$$Y_{m+1}^i = e^{-\gamma}\left[ \int\limits_{m}^{m+1} e^{\gamma (s-m)}d(\sigma_m d\widetilde{W}_s^i)+Y_m^i\right], a.s.$$
We know that $\{\widetilde{W}_s^i\}_{0\le s<1}$ is a Brownian motion on time interval $[m, m+1]$ and that it is independent of $Y_m^i$.

By symmetry, the distributions of $Y_t^i$ and $Y_t^j$ are  the same, so we will  just write  $Y_t^i$.

We thus have $Y_{m+1}^i = e^{-\gamma}(Y_m^i \bigoplus X_m^i)$, where 
\begin{equation}\label{x.part}
X_m^i = \int\limits_{m}^{m+1} e^{\gamma (s-m)}d(\sigma_m d\widetilde{W}_s^i)
\end{equation}
 is a normal variable with distribution $\mathcal N\left(0,\sigma^2_m (2\gamma)^{-1}
\left(e^{2\gamma}-1\right)\right),$ and $Y_0^i = 0$. Note  that $Y_m^i$ and $X_m^i$ (or more generally, $X_m$ and $Y_m$) are  independent, as $X_m$ is defined by a stochastic integral on $[m,m+1]$.

% % % 
% % %  The Variance theorem.
% % %

\begin{proposition} \label{Variance Theorem} 
For any $m\ge 0$,
$$\mathtt{Var}(Y^i_{m+1})=\sigma^2(Y_{m+1}^{i}) = \frac{e^{2\gamma}-1}{2\gamma}\,e^{-2m\gamma}\left(\frac{1-e^{2m\gamma}}{1-e^{2\gamma}} - \frac{1-(\frac{1}{2}e^{2\gamma})^m}{2-e^{2\gamma}}\right).$$ 
\end{proposition} 

\begin{proof}

\noindent Since $Y_{m+1} = e^{-\gamma}(Y_m \bigoplus X_m)$ and $Y_0 = 0$,  we can compute the variance of $Y_{m+1}$ using recursion, due to the 
independence of $Y_m$ and $X_m$. First, we have
\begin{eqnarray*}
\sigma^2(Y_{m+1})&=& e^{-2\gamma}\left(\sigma^2(Y_{m}) + \sigma^2(X_{m})\right) \\   &=& e^{-2\gamma} \left(\sigma^2(Y_{m}) + \frac{1}{2\gamma}\left(1-\frac{1}{2^m}\right)
(e^{2\gamma}-1)\right),
\end{eqnarray*}
 and $\sigma^2(Y_0) = 0$. 
~\\
Thus, for convenience, we denote $a_m:=\sigma^2(Y_m) $ and $$b_m := \sigma^2(X_m) = \frac{1}{2\gamma}\left(1-\frac{1}{2^m}\right)\left(e^{2\gamma}-1\right).$$
~\\
Since $a_{m+1} = e^{-2\gamma}(a_m + b_m),$ we have the following recursion:
\begin{align*}
a_{m+1}        &= e^{-2\gamma}(e^{-2\gamma}(a_{m-1} + b_{m-1}) + b_m)\\
        &= (e^{-2\gamma})^2a_{m-1}+(e^{-2\gamma})^2b_{m-1}+e^{-2\gamma}b_m\\
        &= (e^{-2\gamma})^3a_{m-2}+(e^{-2\gamma})^3b_{m-2}+(e^{-2\gamma})^2b_{m-1}+e^{-2\gamma}b_m\\
        &= (e^{-2\gamma})^{m+1}a_{0}+(e^{-2\gamma})^{m+1}b_{0}\\
        &\ \ \ \ \ +\cdots+(e^{-2\gamma})^3b_{m-2}+(e^{-2\gamma})^2b_{m-1}+e^{-2\gamma}b_m.\\
\end{align*}

\noindent We have $a_0 = b_0 = 0$, and thus, for $m\ge 1$,
\begin{eqnarray*}a_{m+1}&=&\sum_{n=1}^{m}(e^{-2\gamma})^{m+1-n}\frac{1}{2\gamma}\left(1-\frac{1}{2^n}\right)\left(e^{2\gamma}-1\right)\\
&=& \frac{e^{-2m\gamma}(1-e^{-2\gamma})}{2\gamma}\left(\sum_{n=1}^{m}e^{2n\gamma} - \sum\limits_{n=1}^{m}\left(\frac{e^{2\gamma}}{2}\right)^n\right).
\end{eqnarray*}
Summing the series,
$$a_{m+1} = \frac{e^{2\gamma}-1}{2\gamma}\,e^{-2m\gamma}\left[\frac{1-e^{2m\gamma}}{1-e^{2\gamma}} - \frac{1-(\frac{1}{2}e^{2\gamma})^m}{2-e^{2\gamma}}\right],\ m\ge 1.$$
Using the definition of $a_{m+1}$, the proof is complete.
\end{proof}
We now need to analyze the covariance between $Y_m^i$ and $Y_m^j$ for $1\le i,j\le 2^m$ and $m\ge 1$. 
Here the notion of the {\it most recent common ancestor} (MRCA) becomes important.
\begin{definition}[MRCA]\rm 
Consider $m\ge 0$ and the particles $Y_m^i,Y_m^j, 1\le i\neq j\le 2^m$. A common ancestor is a particle considered at some integer  time  interval $[a-1,a)$ with $0<a<m$, which is an ancestor to both $Y_m^i$ and $Y_m^j$. The MRCA of the particles is a common ancestor such that $a$ is maximal. This $a$ is then called the {\it splitting time} of the particles. 

Note that in this definition, we consider  particles living in different  
time intervals $[m,m+1)$ as different particles.
\end{definition}
As  different pairs of particles $(Y_m^i,Y_m^j)$ with different MRCA's may have different covariances, we need to take into account the MRCA  of $(Y_m^i,Y_m^j)$. Hence, with splitting time $a>0$, we will write $(Y_{m,a}^i,Y_{m,a}^j)$ for $(Y_m^i,Y_m^j)$.

% % % 
% % %  The covariance theorem.
% % % 
\begin{theorem}[Covariance bound]
\label{Cov theorem}
For  $m\ge 1$ and $ 1\le i\neq j\le 2^m$, $$\mathtt {Cov}(Y_{m,a}^j,Y_{m,a}^i )<Cm(e^{-2\gamma(m-a)}+2^{-m}),$$where $C$ is a constant which only depends on 
$\gamma$.
\end{theorem} 

\begin{proof}
\noindent Since  $Y_{m+1}^i = e^{-\gamma}(Y_m^i \bigoplus X_m^i)$ for any `relative particle,' it follows by recursion that 
$$Y_{m}^i = e ^{-\gamma(m-a)}Y_a^i + \sum\limits_{n=a}^{m-1}e^{-\gamma(m-n)}X_n^i.$$
Here $X_n^i$ is the process as in \eqref{x.part} on the unit time interval $[n,n+1]$. From the preceding formula we can easily  compute 
$\mathtt {Cov}(Y_m^i,Y_m^j)$, once we  know the corresponding $\mathtt {Cov}(X_n^i,X_n^j)$ for $n=a,a+1,...,m-1$. For the $i$th particle ($i$ can be replaced by $j$, of course), 
~\\
$$X_n^i = \int\limits_{n}^{n+1}e^{\gamma(s-n)} \mathrm{d}(\sigma_m \widetilde{W}_s^i),$$ 
where $\sigma_n \widetilde{W}_s^i= -\frac{1}{2^n}\bigoplus_{k \ne i}W_s^k \oplus \sigma_n^2 W_s^i$ and $\{W_s^k, k = 1,2,3,\cdots,2^n\}$, as well as $\widetilde{W}^i$  are standard independent Brownian motions on the unit time interval $[n,n+1]$.

Using that the $\{W_s^k, k = 1,2,\cdots, 2^n\}$ are  independent, $\mathtt {Cov}(X_n^i,X_n^j)$ can now be computed. Indeed, the stochastic integrals $X_n^i,X_n^j$ are just linear combinations of the stochastic integrals $A_k := \int\limits_{n}^{n+1} e^{\gamma(s-n)} \, \mathrm{d} W_s^k, k = 1,2,3,\cdots$, that is,
\begin{eqnarray*}
& X_n^i = -\frac{1}{2^n}\sum_{k \ne i} A_k + \sigma_n A_i; \\
&X_n^j = -\frac{1}{2^n}\sum_{k \ne j} A_k + \sigma_n A_j.
\end{eqnarray*}
Note that, by the 
It\^{o}-isometry,
$$\mathtt {Var}(A_k) = \frac{e^{2\gamma}-1}{2\gamma}.$$
Therefore, using the independence of the $A_k$'s,
\begin{align}
&\mathtt {Cov}(X_n^i,X_n^j)\nonumber \\
        &\ = \mathtt {Cov}\left(-\frac{1}{2^n}\sum_{k \ne i} A_k + \left(1-\frac{1}{2^n}\right)A_i,-\frac{1}{2^n}\sum_{k \ne j} A_k + \left(1-\frac{1}{2^n}\right)A_j\right)\nonumber\\
        &\ = \frac{1}{2^{2n}}\sum\limits_{k \ne i,j}\mathtt {Var}(A_k) - \frac{1}{2^n}\left(1-\frac{1}{2^n}\right)\mathtt {Var}(A_i) - \frac{1}{2^n} \left(1-\frac{1}{2^n}\right) \mathtt {Var}(A_j)\nonumber \\
        &\ =\left ( \frac{1}{2^{2n}}\left(2^n-2\right) - \frac{2}{2^n}\left(1-\frac{1}{2^n}\right)\right)\frac{e^{2\gamma}-1}{2\gamma}\nonumber \\
        &\ = -\frac{1}{2^n} \frac{e^{2\gamma}-1}{2\gamma}.\nonumber \\
\end{align}\label{hosszu}
Since  $Y_{m}^i = e ^{-\gamma(m-a)}Y_a^i + \sum\limits_{n=a}^{m-1}e^{-\gamma(m-n)}X_n^i$, and the $\{X_k^i\}$
are independent for different $k$'s and finally,  since $Y_a^i = Y_a^j$, it follows that\footnote{In this formula, $i>2^k$ may occur. Then $X_k^i$ stands for the normal distribution associated with the ancestor of $i$ 
at the time $k$. The normal distribution is defined in the same way as for $X_n^i$.}
\begin{equation}\label{cov.Ys}
\mathtt {Cov}(Y_m^i,Y_m^j) = e^{-2\gamma(m-a)}\mathtt {Var}(Y_a^i) + \sum\limits_{k=a}^{m-1}e^{-2\gamma(m-k)} \mathtt {Cov}(X_k^i,X_k^j).
\end{equation}
According to Proposition \ref{Variance Theorem},  $$\mathtt {Var}(Y_a^i) = \frac{e^{2\gamma}-1}{2\gamma}\,e^{-2\gamma(a-1)}\left[\frac{1-e^{2\gamma(a-1)}}{1-e^{2\gamma}} - \frac{1-(\frac{1}{2}e^{2\gamma})^{a-1}}{2-e^{2\gamma}}\right].$$ Using this along with \eqref{hosszu} and \eqref{cov.Ys}, one obtains that
$$\mathtt {Cov}(Y_m^i,Y_m^j) =\frac{e^{2\gamma}-1}{2\gamma}\,\Bigg(e^{-2\gamma(m-1)}\left[\frac{1-e^{2\gamma(a-1)}}{1-e^{2\gamma}} - \frac{1-(\frac{1}{2}e^{2\gamma})^{a-1}}{2-e^{2\gamma}}\right]$$
$$+  \sum\limits_{k=a}^{m-1}e^{-2\gamma(m-k)} \left(-\frac{1}{2^k}\right)\Bigg).$$
It then follows that 
$$ \mathtt {Cov}(Y_m^i,Y_m^j) < C_0(I_1+I_2),$$
where $$I_1 = e^{-2\gamma(m-1)}\left[\frac{1-e^{2\gamma(a-1)}}{1-e^{2\gamma}} - \frac{1-(\frac{1}{2}e^{2\gamma})^{a-1}}{2-e^{2\gamma}}\right],$$
 $$I_2 =\sum\limits_{k=a}^{m-1}e^{-2\gamma(m-k)} \left(-\frac{1}{2^k}\right),$$ and $C_0 =C_0(\gamma):=\frac{e^{2\gamma}-1}{2\gamma}.$
 As $\gamma>0$, one has $|I_1| < C_1e^{-2\gamma(m-a)}$, where $C_1 $ is a constant which only depends on $\gamma$. 
 Moreover, $I_2 = e^{-2m\gamma}\sum\limits_{k=a}^{m-1}\left(-(\frac{e^{2\gamma}}{2})^k \right).$ Then, by an easy computation,  
 $|I_2| < C_2(m\,e^{-2m\gamma} + \frac{1}{2^m})$, were $C_2$ is also a constant which only depends on $\gamma$.
 
Hence, $\mathtt {Cov}(Y_m^i,Y_m^j)$ is bounded from above by $$C_0\left(C_1e^{-2\gamma(m-a)} + C_2\left(m\,e^{-2m\gamma} + \frac{1}{2^m}\right)\right) < C\,m(e^{-2\gamma(m-a)} + 2^{-m} ),$$
where $C$ is a constant that only depends on $\gamma$.\end{proof}

\medskip
 Recall that our goal is to prove the existence of  $\lim\limits_{m\to\infty}2^{-m}\sum\limits_{i=1}^{2^m}1_B(Y_m^i)$ and identify it. To achieve this, we will use the Borel-Cantelli Lemma in conjunction with the Chebysev inequality, and  so we need to estimate $\mathtt {Var}\left(2^{-m}\sum\limits_{i=1}^{2^m}1_B(Y_m^i)\right)$.
Clearly,
$$\mathtt {Var}\left(\sum\limits_{i=1}^{2^m}1_B(Y_m^i)\right) = \sum\limits_{i=1}^{2^m}\mathtt {Var}(1_B(Y_m^i)) + \mathop{\sum_{i=1}^{2^m}\sum_{j=1}^{2^m}}_{i\ne j}\mathtt {Cov}(1_B(Y_m^i),
1_B(Y_m^j)),$$
and
$$\mathtt {Var}(1_B(Y_m^i)) = P(Y_m^i\in B) - P^2(Y_m^i \in B) \le  1/4.$$

\noindent We  need to compute $\mathtt {Cov}(1_B(Y_m^i),1_B(Y_m^j))$ for $i\neq j$, and to do that, we need to analyze  the splitting time (still denoted  by $a$) for the particles 
$Y_m^i$ and 
$Y_m^j$. The following lemma on joint normal distribution will be useful.

% % % Cov control lemma 
\begin{lemma}[Covariance for indicators]
\label{Cov control}
Let  $(X,Y)$ be a joint normal vector such that its marginals $X$ and $Y$ are standard normal, and denote $\rho:=\mathtt {Cov}(X,Y)$. Then there exists an absolute  constant $C>0$ such that $$\mathtt {Cov}(1_B(X),1_B(Y)) \le C |\rho|$$ holds for all Borel sets $B\subset \mathbb R$ and all $|\rho| < 1/2$.
\end{lemma}
\begin{proof}
\noindent Plugging in the joint and marginal densities  $$\frac{1}{2\, \pi\sqrt{1-\rho^2}}e^{-\frac{x^2+y^2+2\rho xy}{2(1-\rho^2)}};\ \frac{1}{\sqrt{2\pi}} e^{-\frac{x^2}{2}},$$ one obtains
\begin{eqnarray*}
\psi(\rho,B)&:=&\mathtt {Cov}(1_B(X),1_B(Y)) \\
&=&  P(X\in B, Y \in B) - P(X\in B)\cdot P(Y\in B)\\
&=& \frac{1}{2 \pi} \iint_{B\times B} \frac{1}{\sqrt{1-\rho^2}}e^{-\frac{x^2+y^2+2\rho xy}{2(1-\rho^2)}} - e^{\frac{-x^2-y^2}{2} }\, \mathrm{d}x\mathrm{d}y\\
&=:& \frac{1}{2 \pi} \iint_{B\times B} f(x,y,\rho)\, \mathrm{d}x\mathrm{d}y.
\end{eqnarray*}
Clearly, 
$\psi(0,B) = 0$. Since $f\in C^{\infty}(\mathbb R\times\mathbb R\times (-1/2,1/2))$, 
$$\psi'(\rho,B) = \frac{1}{2\, \pi}\iint_{B\times B}\frac{\partial f(x,y,\rho)}{\partial \rho} \, \mathrm{d}x\mathrm{d}y,$$
where
\begin{equation*}
\frac{\partial f(x,y,\rho)}{\partial \rho} =
e^{-\frac{x^2+y^2+2\rho xy }{2(1-\rho^2)}}\cdot\left(\frac{\rho}{(1-\rho^2)^{\frac{3}{2}}}  
-\frac{1}{\sqrt{1-\rho^2}} \frac{k(x,y,\rho)}{4(1-\rho^2)^2}\right),
\end{equation*}
with $k(x,y,\rho):=2xy 2(1-\rho^2)+ 4\rho (x^2+y^2+2\rho xy).$
For $|\rho| \le \frac{1}{2}$, we have $ \frac{x^2+y^2}{2}\le x^2+y^2+2\rho xy,$ i.e. $x^2+y^2+4\rho xy\ge 0$, and  thus,
$$ \exp\left\{-\frac{x^2+y^2+2\rho xy }{2(1-\rho^2)}\right\} \le \exp\left\{-\frac{x^2+y^2}{4}\right\}.$$
Also, $k(x,y,\rho) \le 2(x^2 + y^2) + 2(2(x^2 + y^2))\le 6(x^2 + y^2)$.
Hence, with some $C>0$ constant,
$$\left|\frac{\partial f(x,y,\rho)}{\partial \rho}\right| \le Ce^{-\frac{x^2+y^2}{4}}(1+x^2+y^2),\ \forall |\rho|\le 1/2.$$
Consequently,
\begin{eqnarray*}
|\psi'(\rho,B)| &\le\frac{1}{2\, \pi}\iint_{B\times B} Ce^{-\frac{x^2+y^2}{4}}(1+x^2+y^2) \, \mathrm{d}x\mathrm{d}y\\
&\le\frac{1}{2\, \pi}\iint_{\mathbb{R}\times \mathbb{R}} Ce^{-\frac{x^2+y^2}{4}}(1+x^2+y^2) \, \mathrm{d}x\mathrm{d}y\\
&\le C\int_0^{\infty} e^{-r^{2}/4}\,r(r^2+1) dr =:C'<\infty,
\end{eqnarray*}  
yielding
$$|\psi(\rho,B)| = \left|\int\limits_{0}^{\rho} \psi'(s,B) ds\right| \le \int\limits_{0}^{|\rho|} |\psi'(s,B)| ds \le C' |\rho|.$$
\end{proof}
\begin{remark} \rm More generally,  let both $X$ and $Y$ be  $\mathcal{N}(0,\sigma^2)$-distributed. If $(X,Y)$ is joint normal with $\mathtt {Cov}(X,Y) = \rho$,
then we can scale $X$ and $Y$ to use Lemma \ref{Cov control}.
Indeed, $\frac{X}{\sigma}$ and  $\frac{Y}{\sigma}$ are then standard normal variables, and  $\mathtt {Cov}( \frac{X}{\sigma}, \frac{Y}{\sigma}) = \frac{\rho}{\sigma^2}$.
From Lemma \ref{Cov control}, if $\frac{|\rho|}{\sigma^2}<1/2$ ($2|\rho|<\sigma^2$), then
\begin{equation}\label{even.better}
\mathtt {Cov}(1_B(X),1_B(Y))= \mathtt {Cov}\left(1_{\frac{B}{\sigma}}\left(\frac{X}{\sigma}\right), 1_{\frac{B}{\sigma}}\left(\frac{Y}{\sigma}\right)\right)\le \frac{C|\rho|}{\sigma^2}.
\end{equation}$\hfill\diamond$\end{remark}
Returning to the question of covariances with splitting time $a$, note that $\{Y_m^i\}$ are linear combinations of a number of underlying Brownian motions; hence $(Y_m^i,Y_m^j)$ will be joint normal.
From Theorem \ref{Cov theorem}, we know that $\mathtt {Cov}(Y_m^i,Y_m^j) \le C(\gamma)m(e^{-2\gamma(m-a)}+2^{-m})$.  From \eqref{even.better}, we then  have the following corollary.
\begin{corollary}\label{cov.coroll} With some constant $C>0$ (that depends only on $\gamma$),
  $$\mathtt {Cov}\left(1_B(Y_m^i),1_B(Y_m^j) \right)\le \frac{Cm(e^{-2\gamma(m-a)}+2^{-m})}{\sigma^2(Y_m^i)},$$
  provided $2m(e^{-2\gamma(m-a)}+2^{-m})<\sigma^2(Y_m^i).$
   (Of course, $\sigma^2(Y_m^i) = \sigma^2(Y_m^j).$)
\end{corollary}

%\noindent For a fixed $B$ which have a Lebesgue measure $M(B) > 0$, then $s \le \sigma^2(1_{B}(Y_m^i)) \le S$ if $v \le \mathtt {Var}(Y_m^i) \le V$. And 
%$s,S$ depend on $v,V$ and $B$.(Make it as a Lemma?)\\
%~\\
\begin{remark} \rm From Proposition \ref{Variance Theorem}, it follows that $\lim\limits_{m\to \infty}\sigma^2(Y_m^i) = \frac{1}{2\gamma}$. Therefore, if we write $a(m)$ in place of $a$, and $\lim_{m\to\infty} (m-a(m))=\infty$, then the condition in Corollary \ref{cov.coroll} will always be true for large enough $m$'s.
$\hfill\diamond$\end{remark}
\subsection{Strong Law for the relative system; $d=1$} 
Now we are ready to prove the following theorem.

% % % %
% % % % Measure limit theorem
% % % %

\begin{theorem} [SLLN; one dimension]\label{SLLN.thm}
%\label{Measure limit theorem}
Assume that $d=1, \gamma + b>0$ and let   $B\subset \mathbb R^d$ be a Borel set. Then, almost surely, $$\lim\limits_{m\to \infty} \left(2^{-m}\sum\limits_{i=1}^{2^m}1_B(Y_m^i) - P(Y_m^i\in B)\right) = 0.$$ 
\end{theorem}
Here, of course $\lim_{m\to\infty}P(Y_m^i\in B)$ exists. That is,
similarly to  \cite{Englander},  $\lim\limits_{m\to\infty}2^{-m} Y_m(\,\mathrm{d}y) = f(y)\, \mathrm{d}y$ a.s. in the weak topology, where  $f$ is the density  for  $\mathcal{N}\left(0,\frac{1}{2\gamma}\right)$.\\
\begin{proof}
Recall Assumption \ref{nodriftass}: we may and will assume that $b=0$ and $\gamma>0$. 

We need to show that for any $\epsilon > 0$, a.s., only finitely many of the events $$A_m := \left\{\left|2^{-m}\sum\limits_{i=1}^{2^m}1_B(Y_m^i) - E(1_B(Y_m^i))\right| >\epsilon\right\}$$
will occur.
By the Borel-Cantelli Lemma and Chebyshev's inequality, it is enough to show that  $$\sum\limits_{1}^{\infty}P(A_m) \le \epsilon^{-2}\, \sum\limits_{1}^{\infty}\mathtt {Var}\left(2^{-m}\sum\limits_{i=1}^{2^m}1_B(Y_m^i)\right)<\infty.$$ (Note that
$E\left(2^{-m}\sum\limits_{i=1}^{2^m}1_B(Y_m^i)\right) = E(1_B(Y_m^i))$.) Thus, it remains   to  show that
\begin{equation}\label{what.we.need}
\sum\limits_{m=1}^{\infty}\mathtt {Var}\left(2^{-m}\sum\limits_{i=1}^{2^m}1_B(Y_m^i)\right) < \infty.
\end{equation}
To this end, note that 
\begin{eqnarray}&\mathtt {Var}\left(2^{-m}\sum\limits_{i=1}^{2^m}1_B(Y_m^i)\right)= \frac{1}{2^{2m}}
\big( \sum\limits_{i=1}^{2^m}\mathtt {Var}(1_B(Y_m^i))\\ &\ \ \ + \mathop{\sum_{i=1}^{2^m}\sum_{j=1}^{2^m}}_{i\ne j}\mathtt {Cov}(1_B(Y_m^i),
1_B(Y_m^j))\big)\nonumber
\end{eqnarray} 

\noindent Since $\mathtt {Var}(1_B(Y_m^i)) \le 1/2 - 1/4 = 1/4 $, we have 
$$\mathtt {Var}\left(2^{-m}\sum\limits_{i=1}^{2^m}1_B(Y_m^i)\right)  \le \frac{1}{2^{2m}}\left( \frac{2^m}{4} + 2^m\sum\limits_{j\ne i} \mathtt {Cov}(1_B(Y_m^i), 1_B(Y_m^j))\right),$$ 
for a fixed $i$, since $i$ and $j$ are symmetric.

Hence, one needs to analyze $\sum\limits_{j\ne i} \mathtt {Cov}(1_B(Y_m^i),
1_B(Y_m^j))$  for a fixed $i$. We know that $\mathtt {Cov}(1_B(Y_m^i), 1_B(Y_m^j))$ depends on the time $a$ when  the MRCA  of these particles splits. We thus need to distinguish between `close relatives' and other pairs.

Notice that there are $2^k$ particles which have the MRCA at time $m-k$ with $Y_m^i$. From Theorem \ref{Cov theorem},
we know that $$\mathtt {Cov}(Y_{m}^j,Y_{m}^i )<Cm(e^{-2\gamma(m-a)}+2^{-m}),$$ if the MRCA of $i$ and $j$ is $a$.  If $a=a(m)=m/2$, (or just $m-a$ tends to $\infty$), then the righthand side converges to zero as $m\to\infty$.  As already seen,  $\lim\limits_{m\to \infty} \mathtt {Var}(Y_m^i) = \frac{1}{2\gamma}$, and
so we may apply Lemma \ref{Cov control} and the remark following it for  $a=a(m) \le  \frac{m}{2}$ and large $m$. That is, we could choose a large $N$ such that for all $m>N$, the condition
\begin{equation}\label{cov.and.var}
\frac{\mathtt {Cov}(Y_m^i, Y_m^j)}{\mathtt {Var}(Y_m^i)} \le \frac{1}{2},
\end{equation}
is satisfied.

Now, the important point is that  {\it the majority of particle-pairs have $a\le m/2$, that is, they are not `close' relatives.} Indeed,  the number of pairs with $a>m/2$ (close relatives) is $2^m(1+ 2+ 4 +\cdots + 2^{\lfloor\frac{m-1}{2}\rfloor})$. Simple computation yields that  $$2^m(1+ 2+ 4 +\cdots + 2^{\lfloor\frac{m-1}{2}\rfloor}) \le 2^m(2^{\lfloor\frac{m-1}{2}\rfloor + 1}) \le 2^{m+\frac{m}{2} + 1}.$$

As we know, for any pair, we have $\mathtt {Cov}(Y_m^i, Y_m^j) \le \mathtt {Var}(Y_m^i) \le \frac{1}{4}$. Thus, for all of those pairs with $a > m/2$, the total covariance will be controlled by $2^{m+\frac{m}{2} + 1}\cdot  \frac{1}{4} = 2^{m+\frac{m}{2} - 1}$.  Moreover, as discussed above, for the pairs with $a \le m/2$, we may apply  Lemma \ref{Cov control} and the remark following it, yielding
\begin{align*}
\sum\limits_{j\ne i} \mathtt {Cov}(1_B(Y_m^i),1_B(Y_m^j)) &\le \sum\limits_{k=\lfloor\frac{m+1}{2}\rfloor}^{m-1} 2^kCm(e^{-2\gamma k} + 2^{-m}) \\
&\le Cm\sum\limits_{k=1}^{m-1} \left(\left(\frac{2}{e^{2\gamma}}\right)^k + 2^{k-m}\right),
\end{align*}
where $C$ only depends on $N,\gamma,B$.
Consequently,
\begin{align*}
&\mathtt {Var}\left(2^{-m}\sum\limits_{i=1}^{2^m}1_B(Y_m^i)\right)\\
        &\ \ \le \frac{1}{4\times 2^m} +\frac{2^{\frac{m}{2} - 1}}{2^m}+ \frac{Cm}         {2^m}\left(\frac{1-(\frac{2}{e^{2\gamma}})^m}{1-\frac{2}{e^{2\gamma}}} + (1- 2^{-m})  \right)\\
        &\ \ \le \frac{C_1m}{2^m} + \frac{C_2m}{e^{2\gamma m}} + \frac{C_3m^2}{2^m} + \frac{1}{2^{\frac{m}{2}+1}}\\
        &\ \  \le \frac{C_0m}{e^{2\gamma m}} + \frac{C_0m^2}{2^m}+ \frac{1}{2^{\frac{m}{2}+1}}.
\end{align*}

\noindent Here $C_0$ is a constant which only depends on $N,\gamma,B$.
Given that $\gamma>0$ and the three terms on the righthand side are all summable in $m$,
\eqref{what.we.need} holds and the proof is complete.
\end{proof}
\subsection{Strong Law for the relative system; $d>1$}
We now show that the limit in Theorem \ref{SLLN.thm}  holds for any $d\ge 1$.  In fact, the  proof carries through, as  long as the covariance in high dimensions is controlled by its coordinates. The following lemma shows that the covariance between two indicator variables  is controlled by the covariance between the coordinate indicator variables.

\begin{lemma}[Control by coordinates]
\label{Le: con_coordinate}
Consider an open rectangle $B$ in $\mathbb{R}^d$,  that is, $B=B_1\times B_2\times B_3\cdots \times B_d$, where $B_i$ is an open interval in $\mathbb{R}$ for $i=1,2,...,d$. Let $X=(X_1,X_2,\cdots,X_d)$ and $Y=(Y_1,Y_2, \cdots, Y_d)$ be two random vectors in $\mathbb{R}^d$ satisfying that the pairs $(X_1,Y_1),(X_2,Y_2),...,(X_d,Y_d)$ are independent. Then $$|\mathtt{Cov}(1_B(X),1_B(Y))| \le \sum_{i=1}^d |\mathtt{Cov}(1_{B_i}(X_i),1_{B_i}(Y_i))|.$$
\end{lemma}

\begin{proof}
One has
\begin{align*}
&\mathtt{Cov}(1_B(X),1_B(Y))\\ 
&\ = P(X\in B, Y\in B) - P(X\in B) P(Y\in B)\\
  &\ = P(X_1\in B_1,\cdots, X_d\in B_d,Y_1\in B_1,\cdots, Y_d\in B_d)\\
 &\ \ \ \ \ -P(X_1\in B_1,\cdots, X_d\in B_d) P(Y_1\in B_1,\cdots, Y_d\in B_d).
\end{align*}
Using the assumption,
\begin{align*}
P(X_1\in B_1,\cdots,& X_d\in B_d) \\
&=P(X_1\in B_1) P(X_2\in B_2) \cdots P(X_d\in B_d);\\
P(Y_1\in B_1,\cdots,& Y_d\in B_d)\\
&=P(Y_1\in B_1) P(Y_2\in B_2) \cdots P(Y_d\in B_d),
\end{align*}
and
\begin{align*} &P(X_1\in B_1,\cdots, X_d\in B_d,Y_1\in B_1,\cdots, Y_d\in B_d) \\
&=P(X_1\in B_1,Y_1\in B_1) P(X_2\in B_2,Y_2\in B_2) \cdots P(X_d\in B_d, Y_d\in B_d).\end{align*}
Using the shorthands $a_i := P(X_i\in B_i,Y_i\in B_i), b_i := P(X_i\in B_i)  P(Y_i\in B_i)$,  one has $\mathtt{Cov}(1_{B_i}(X_i), 1_{B_i}(Y_i)) = a_i - b_i $, and from the computation above,
$$\mathtt{Cov}(1_B(X),1_B(Y)) = a_1 a_2 \cdots a_d - b_1 b_2 \cdots b_d.$$
Therefore the statement becomes 
$$|a_1 a_2 \cdots a_d - b_1 b_2 \cdots b_d| \le  |a_1-b_1| + |a_2-b_2| + \cdots + |a_d-b_d|.$$
Use that $0\le a_i \le 1$ and $0\le b_i \le 1$ and induction on $d$ as follows.
The statement is true for $d=1$, and if it is true for some $d\ge 1$, then
\begin{align*}
&|a_1 a_2 \cdots a_d a_{d+1} - b_1 b_2 \cdots b_d b_{d+1}| \\
=& |(a_1 a_2 \cdots a_d -\cdots - b_1 b_2 \cdots b_d)a_{d+1}+(a_{d+1}-b_{d+1})b_1b_2...b_d|\\
\le &|(a_1 a_2 \cdots a_d -\cdots - b_1 b_2 \cdots b_d)a_{d+1}|+|(a_{d+1}-b_{d+1})b_1b_2...b_d|\\
\le &|(a_1 a_2 \cdots a_d -\cdots - b_1 b_2 \cdots b_d)|+|a_{d+1}-b_{d+1}|\\
\le &|a_1-b_1| + |a_2-b_2| + \cdots + |a_d-b_d|+|a_{d+1}-b_{d+1}|,
\end{align*}
and so it is also true for $d+1$.
\end{proof}
The above lemma immediately yields the following corollary.
\begin{corollary}\label{SLLN.multidim}
Theorem \ref{SLLN.thm}  holds for $d>1$ as well. 
\end{corollary}
% % % % 
\section{The distribution of the particle system}
\noindent Now we have collected enough information to describe the large time behavior  the system as a whole. 
\subsection{Preparation}
Below we describe the system's behavior as it depends on the parameters $\gamma, b$. 
The statements about the large time behavior of the branching particle system will follow from the behavior of the center of mass (Theorem \ref{Th:center} and Theorem \ref{escape}), that of the relative system (Theorem \ref{SLLN.thm} and Corollary \ref{SLLN.multidim}), and finally, from the following proposition.
\begin{proposition}[Independence] \label{indep.prop}
The tail $\sigma$-algebra $\mathcal{T}$ of $\overline{Z}$ is independent of the relative system $Y$.
\end{proposition}
\begin{proof}
The proof of this proposition is exactly the same as the corresponding proof of Lemma 14 in \cite{Englander}.
\end{proof}

Recall that  $Z(\mathrm{d}y)$ denotes the discrete measure-valued process corresponding to the interacting branching  system.
The following notion will be important.
\begin{definition}[Local extinction]\rm 
We say that $Z$ suffers \emph{local extinction}, if 
\begin{equation}\label{loc.ext. def}
Z_n(\mathrm{d}y)\overset{v}\Rightarrow 0,\  \text{as}\ n\to\infty,\ a.s.
\end{equation}
\end{definition}
Since $Z$ is a discrete particle system, \eqref{loc.ext. def} is tantamount to the property that there exists an almost surely finite random time $T$ such that $$P(Z_n(B)=0,\ \forall n\ge T,\ \forall B\subset \mathbb R^d\ \text{ball})=1.$$

\subsection{Inward drift in the motion component}
We now turn to the first results about the behavior of $Z$, starting with the case of an inward drift ($b>0$).
In fact, we distinguish between three further sub-cases.

\medskip
\noindent\textbf{Case 1: $b>0,\ b+\gamma > 0$}.\\
\noindent As we have demonstrated, the center of mass converges to zero as $t\to\infty$, and the relative system will be an inward O-U process with parameter $\gamma + b$. Putting  Theorem \ref{SLLN.thm} and Corollary \ref{SLLN.multidim} together with the a.s. converge of the C.O.M. (Theorem \ref{Th:center}), and finally, with Proposition \ref{indep.prop}, we arrive at the following conclusion.
\begin{theorem} Assume that $b>0$, and $b+\gamma > 0$. Then
$$2^{-n}Z_n(\mathrm{d}y)\Rightarrow \left( \frac{\gamma+b}{\pi}\right)^{d/2} \exp(-(\gamma + b)|y|^2)\, \mathrm{d}y,$$ almost surely.
\end{theorem}
\begin{remark}\rm
Since $\mathcal{R}$ is a countable family, the weak limit in the previous theorem is actually   equivalent to the statement that for all $B\in\mathcal{R}$,
$$\lim_{n\to\infty}2^{-n}Z_n(B)= \int_B\left( \frac{\gamma+b}{\pi}\right)^{d/2} \exp(-(\gamma + b)|y|^2)\, \mathrm{d}y,\ a.s.$$ 
 (See the appendix for more elaboration.)$\hfill\diamond$
 \end{remark}
 \begin{example}[Non-interactive branching O-U process]\label{non.inter.OU}\rm 
Consider the case $\gamma=0, b>0$, that is, the case of a non-interacting branching (inward) O-U process with parameter $b$. The proof goes through in this case as well, and we obtain that for all $B\in\mathcal{R}$,
$$\lim_{n\to\infty}2^{-n}Z_n(B)= \int_B\left( \frac{b}{\pi}\right)^{d/2} \exp(-b|y|^2)\, \mathrm{d}y,\ a.s.,$$ 
complementing the exponential-clock results in \cite{EHK2010}.
\end{example}

\noindent\textbf{Case 2: $b>0,\ b+\gamma = 0.$ }\\
\begin{theorem}   In this case, 
\begin{equation}\label{Watanabe.type}
\lim_{n\to\infty}C n^{d/2}2^{-n}Z_n(B)=\mathtt{Leb} (B),\ a.s.,
\end{equation}
for any bounded Borel set $B$,
where $\mathtt{Leb}$ denotes Lebesgue measure. Here $C:=(2\pi)^{d/2}.$ 
\end{theorem} 
\begin{proof}
The proof is based on an argument which `switches off' the interaction.
In order to accomplish this, we are going to utilize the lemma on interchangeability (Lemma \ref{interchange}). Namely, we match the relative system with that of another system without interaction.

This other  system  is the one with $b=\gamma=0$ (branching Brownian motion without interaction). As far as the behavior of this second system is concerned,
it is well known (see \cite{W1967,B1992}), that \eqref{Watanabe.type} holds. 

Even though in \cite{W1967,B1992}, the decomposition into C.O.M. and a relative system was not considered, we do that now.
That is useful, because by Lemma \ref{interchange}, the relative system is the same for the two processes, even though the behavior of the C.O.M. is not: 
for the original system it converges to the origin almost surely (Theorem \ref{Th:center}), and for the non-interacting BBM it has an almost sure (Gaussian) limit (see \cite{Englander}).

Now use  the fact that Lebesgue measure is translation invariant. In both systems, one has
$$Z_t(B)=Y_t(B+\overline{Z}_t),$$
hence
$$C n^{d/2}2^{-n}Z_n(B)=C n^{d/2}2^{-n}Y_n(B+\overline{Z}_n).$$ By conditioning on the almost sure limit of $\overline{Z}_n$, and using Proposition \ref{indep.prop}, the translation invariance of the Lebesgue measure yields \eqref{Watanabe.type} for the {\it relative} system in the $b=\gamma=0$ case.

 But then, by Lemma \ref{interchange}, the same holds for the relative system in the original model.  Since the C.O.M.  converges to the origin almost surely for the original model, using Proposition \ref{indep.prop}, we conclude that  the scaling limit \eqref{Watanabe.type} also holds for the original system.
\end{proof}

\medskip
\noindent\textbf{Case 3:  $b>0,\ b+\gamma < 0.$ }\\
\noindent As the relative system behaves asymptotically like an outward O-U process, we have a conjecture  similar to the one  in \cite{Englander}. In our case, however,
the center of mass tends to $0$ as $t\to\infty$. Thus, the conjecture will take the following form:
\begin{conjecture}\label{sejtes}
The following dichotomy holds for the long term behavior of the process:
\begin{enumerate}
\item If $\frac{\log2}{d} \le|b+\gamma|$, then $Z$ suffers local extinction.
\item If $\frac{\log2}{d} > |b+\gamma|$, then 
$$ 2^{-n}e^{d|b+\gamma| n}Z_n(\mathrm{d}y) \overset{v} \Rightarrow \mathrm{d}y.$$
\end{enumerate}
\end{conjecture}
\begin{remark}\rm 
The intuitive explanation of the dichotomy in the conjecture is as follows. Even though the motion component has a strong inward component (forcing the center of mass to tend to the origin, according to Theorem \ref{Th:center}), this is offset by the even stronger repulsion term. 

This combined effect is then competing with the mass creation (the `rate' of mass creation in this case can be taken $\log 2$): if mass creation is stronger, then the Law of Large Numbers is still in force; otherwise the mass creation is no longer able to compensate the fact that particles are `being pushed away.'$\hfill\diamond$
\end{remark}

\subsection{Outward drift in the motion component}
This case is more difficult than the case of the inward drift. The result below is quite natural once the decomposition of the process (C.O.M. plus relative system) is established, however, it may be somewhat  surprising if one is just given the definition of the model with the pairwise particle interactions.

\medskip
\noindent\textbf{Case 4:} $b<0$ (Outward drift)

In this case,  according to Theorem \ref{escape}, the center of mass converges to infinity a.s. as $t\to\infty$, and so the question is, intuitively, whether this effect will be compensated by the large number of particles.

The next result says that for outward drift and attraction, the system {\it always} exhibits local extinction, irrespective of the relationship  between  the drift size $|b|$ and the attraction parameter $\gamma$.
\begin{theorem} For $b<0$ and $\gamma>0$ (outward O-U with attraction), there is local extinction: $$Z_n(\mathrm{d}y)\overset{v}\Rightarrow 0 \ a.s.$$ 
\end{theorem}
\begin{remark} As far as the relative system's behavior is concerned, that is of course given by Theorem \ref{SLLN.thm} and Corollary \ref{SLLN.multidim}.$\hfill\diamond$
\end{remark}
\begin{proof}
The proof is based on  comparing the speed of the C.O.M. with that of the relative system. 
Let $M^{\text{rel}}_t$ denote the radius of the smallest ball containing the support of $Y_t$ for $t\ge 0$. It is clear that the relation
\begin{equation}\label{implies.le}
\lim_{t\to\infty}\frac{M^{\text{rel}}_t}{|\overline{Z}_t|}=0,\ \text{a.s.},
\end{equation}
proven below,
 implies local extinction; we will verify it by distinguishing between three sub-cases.

\begin{enumerate}
\item[{\sf(a)}] When $b=\gamma=0$, the non-interactive branching Brownian motion will be denoted by $\mathcal{Z}$ ($b=\gamma=0$). Let $Y^{\text{non-int}}$ be the relative system for $\mathcal{Z}$ and let $M^{\text{non-int}}_t$ be the radius of the smallest ball containing the support of $Y^{\text{non-int}}.$

Using the Borel-Cantelli lemma, and then bounding the union by the sum, for $c>\sqrt{2\log 2}$ we obtain 
$$P( \mathcal{Z}_n(\mathbb R^d\setminus B(0,cn))>0\ \text{i.o.})=0.$$ Indeed, this follows from the fact that 
\begin{align*}
\sum_n \sum_{i=1}^{2^n} P( \mathcal{Z}^{i}_{n}(\mathbb R^d\setminus B(0,cn))>0)<\infty,
\end{align*}
which, in turn, follows from the estimate\footnote{Use Brownian scaling and the Gaussian upper tail estimate.} 
$$P( \mathcal{Z}^{i}_{n}(\mathbb R^d\setminus B(0,cn))>0)=\exp (-c^2 n/2+o(n)),$$
where $\{\mathcal{Z}^{i}_{n}; 1\le i\le 2^n\}$ are the particles forming $\mathcal{Z}_{n}$. 

(The asymptotic `speed of the support' of $\mathcal{Z}$ is actually exactly $\sqrt{2\log 2}$, but we don't need this fact here. Cf. \cite{Ky2005}.) 

We now claim that  $b+\gamma = 0$ implies
\begin{equation}\label{wenowclaim}
M^{\text{rel}}_t=\mathcal{O}(t), t\to\infty.
\end{equation} Indeed, recall that by $\Delta$-transform invariance, $$Y=Y^{\text{non-int}}\ \text{in law}.$$ 
But $M^{\text{non-int}}_t=\mathcal{O}(t),$ as $t\to\infty$, because it is well known that
 $$\max\{|z|\mid z\in \mathrm{supp}(\mathcal{Z}_t)\}=\mathcal{O}(t),$$ and because of the existence of  $\lim_{t\to\infty}\overline{\mathcal{Z}_t}$  a.s..

Turning back to $Z$, if $b+\gamma = 0$,  then \eqref{wenowclaim} and the fact that $\overline{Z}_t$ escapes to infinity exponentially fast (according to our Theorem \ref{escape}) yields
\eqref{implies.le}.

\item[{\sf(b)}] If $b+\gamma > 0$, then by $\Delta$-transform invariance, for the relative system, we may assume that it is non-interacting and the particles are performing inward Ornstein-Uhlenbeck motions with parameter $b+\gamma > 0$. The proof is then the same as in ${\sf(a)}$, given that $P( \mathcal{Z}^{i}_{n}(\mathbb R^d\setminus B(0,cn))>0)$  is even smaller now.

\item[{\sf(c)}] Finally, for the case when $b+\gamma < 0$ and $ \gamma > 0$, we still have \eqref{implies.le}. To see this, recall first that  the logarithmic escape rate  of $\overline{\mathcal{Z}_t}$ is $-bt$. On the other hand, a calculation similar\footnote{There the clock is exponential, whereas here it is unit time.} to the one in Example 11 in \cite{EHK2010} reveals that the logarithmic rate of spread of the relative system is $-(b+\gamma)t$, which yields
\eqref{implies.le}. (Since the probability that at least
one particle is present in a set $B$ is trivially dominated by the expected particle number in that set, that is, $P( \mathcal{Z}_{n}(B)>0)\le E( \mathcal{Z}_{n}(B))$, the calculation reduces to computing certain expectations. In our case it is even easier than in \cite{EHK2010}, as the total population size is deterministic.)
\end{enumerate}
Hence, in all three cases, local extinction occurs.
\end{proof}
We conclude with posing an open problem:

\medskip
\noindent{\bf Open problem} (outward O-U with repulsion).\ 
Describe the large-time behavior of the relative/global system for $b,\gamma<0$. (As far as the relative system's behavior  is concerned, the conjecture is the same as in Conjecture \ref{sejtes}.)

\section{The behavior of $\overline{Z}$ for a drift $b(\cdot)$ bounded between positive constants}
So far we have been working under the assumption that the drift is linear: $b(x)=bx$. Assume now instead, that the drift  satisfies that $$0<a<b(x)<b.$$  
For simplicity, we still start with $d = 1$.  Given that
 $$\mathrm{d}Z^i_t =  \mathrm{d}W^i_t + b(Z^i_t)\, \mathrm{d}t + \frac{1}{2^m} \sum_{j=1}^{2^m}\gamma (Z^j_t-Z^i_t)\, \mathrm{d}t,$$ the motion of the center of mass $\overline{Z_t}$ satisfies the equation
 $$\mathrm{d}\overline{Z_t} = \frac{1}{2^m} \sum_{i=1}^{2^m}\mathrm{d}W^i_t + \frac{1}{2^m} \sum_{i=1}^{2^m} b(Z^i_t)\, \mathrm{d}t.$$ 
 As $a<b(x)<b$, we know that $a<\frac{1}{2^m} \sum_{i=1}^{2^m} b(Z^i_t)<b$. We thus have
 $$  \frac{1}{2^m} \sum_{i=1}^{2^m}\mathrm{d}W^i_t + a \, \mathrm{d}t < \mathrm{d}\overline{Z_t} < \frac{1}{2^m} \sum_{i=1}^{2^m}\mathrm{d}W^i_t + b\, \, \mathrm{d}t.$$
Integration on both sides yields
 $$ \frac{1}{2^m} \sum_{i=1}^{2^m}W^i_t + at < \overline{Z_t} <\frac{1}{2^m} \sum_{i=1}^{2^m}W^i_t + bt.$$
 Since (as we have discussed above) $\frac{1}{2^m} \sum_{i=1}^{2^m}W^i_t$ is a  Brownian motion being slowed down, and $b>a>0$, the center $\overline{Z_t}$ will tend to $+\infty$ with an `essentially constant speed' a.s. (For $a<b(x)<b<0$, it tends to $-\infty$ with essentially constant speed a.s.) More precisely, we have the following result.
 \begin{theorem}
 Assume  that the drift  satisfies $0<a<b(x)<b.$ Then
 $$0<a<\liminf_{t\to\infty} t^{-1}\overline{Z_t}\le \limsup_{t\to\infty}  t^{-1}\overline{Z_t}<b.$$
\end{theorem}
Finally,  a similar result holds for $d > 1$, when replacing $|\overline{Z_t}|$ with $\overline{Z_t}$, as one can consider the statement coordinate-wise. 
 
\bigskip
{\bf Acknowledgements:} We are indebted to the referees and to the associate editor for their close reading of the manuscript and for their valuable suggestions.

\end{document}